\documentclass[11pt, twoside]{article}
\usepackage{amssymb, amsmath, amsthm}
\usepackage[left=25mm, right=25mm, bottom=30mm]{geometry}
\usepackage{inputenc}
\usepackage{fancyhdr}
\usepackage{tikz}
\usetikzlibrary{positioning}
\usepackage{titling}
\usepackage{url}
\usepackage{hyperref}
\predate{}
\postdate{}
\usepackage{lipsum}
\newtheorem{theorem}{Theorem}
\newtheorem{lemma}[theorem]{Lemma}
\newtheorem{proposition}[theorem]{Proposition}
\newtheorem{claim}{Claim}

\newtheorem{question}[theorem]{Question}
\newtheorem{conjecture}[theorem]{Conjecture}
\begin{document}
\newcommand{\Addresses}{{
\bigskip
\footnotesize

\medskip

Maria-Romina~Ivan, \textsc{Department of Pure Mathematics and
Mathematical Statistics, Centre for Mathematical Sciences, Wilberforce
Road, Cambridge, CB3 0WB, UK.}\par\nopagebreak\textit{Email address:}
\texttt{mri25@dpmms.cam.ac.uk}}}
\pagestyle{fancy}
\fancyhf{}
\fancyhead [LE, RO] {\thepage}
\fancyhead [CE] {MARIA-ROMINA IVAN}
\fancyhead [CO] {MINIMAL DIAMOND-SATURATED FAMILIES}
\renewcommand{\headrulewidth}{0pt}
\renewcommand{\l}{\rule{6em}{1pt}\ }
\title{\Large\textbf{MINIMAL DIAMOND-SATURATED FAMILIES}}
\author{MARIA-ROMINA IVAN}
\date{}
\maketitle
\begin{abstract}
For a given fixed poset $\mathcal P$ we say that a family of subsets of $[n]$ is $\mathcal P$-saturated if it does not contain an induced copy of $\mathcal P$, but whenever we add to it a new set, an induced copy of $\mathcal P$ is formed. The size of the smallest such family is denoted by $\text{sat}^*(n, \mathcal P)$.\par For the diamond poset $\mathcal D_2$ (the two-dimensional Boolean lattice), Martin, Smith and Walker proved that $\sqrt n\leq\text{sat}^*(n, \mathcal D_2)\leq n+1$. In this paper we prove that $\text{sat}^*(n, \mathcal D_2)\geq (4-o(1))\sqrt n$. We also explore the properties that a diamond-saturated family of size $c \sqrt n$, for a constant $c$, would have to have.
\end{abstract}
\section{Introduction}
We say that a poset $(\mathcal Q, \preccurlyeq')$ contains an
\textit{induced copy} of a poset $(\mathcal P, \preccurlyeq)$ if there exists an injective function $f:\mathcal P\rightarrow\mathcal Q$ such that $(\mathcal P, \preccurlyeq)$ and $(f(\mathcal P), \preccurlyeq')$ are isomorphic. In this paper we will only consider posets consisting of subsets of $[n]$ with the partial order given by inclusion. For a fixed poset $\mathcal P$ we call a family $\mathcal F$ of subsets of $[n]$ $\mathcal P$-\textit{saturated} if $\mathcal F$ does not contain an induced copy of $\mathcal P$, but for every subset $S$ of $[n]$ such that $S\notin\mathcal F$, the family $\mathcal F\cup \{S\}$ does contain such a copy. We denote by $\text{sat}^*(n, \mathcal P)$ the size of the smallest $\mathcal P$-saturated family of subsets of $[n]$. In general we refer to $\text{sat}^*(n, \mathcal P)$ as the \textit{induced saturation number} of $\mathcal P$.\par Saturation for posets was introduced by Gerbner, Keszegh, Lemons, Palmer, P{\'a}lv{\"o}lgyi and Patk{\'o}s \cite{gerbner2013saturating}, although this was not for \textit{induced} saturation. We refer the reader to the textbook of Gerbner and Patk{\'o}s \cite{gerbner2018extremal} for a nice introduction to the area, as well as to the work of Morrison, Noel and Scott \cite{Morrison2014OnSK} for more results. We also remark that considerable efforts have been made on related extremal problems in poset theory, such as the recent work of Keszegh, Lemons, Martin, P\'{a}lv\"{o}lgyi and Patk\'{o}s \cite{KESZEGH2021}.\par The diamond poset, which we denote by $\mathcal D_2$, is the 4 point poset with one minimal element, one maximal element and two incomparable elements as shown in the picture below. Equivalently, $\mathcal D_2$ is the two-dimensional Boolean lattice.
\begin{center}
\resizebox{3 cm}{2.5 cm}{
\begin{tikzpicture}
\node (top) at (10,2) {$\bullet$}; 
\node (left) at (8,0) {$\bullet$};
\node (right) at (12,0) {$\bullet$};
\node (bottom) at (10,-2) {$\bullet$};
\draw (top) -- (left) -- (bottom) -- (right) -- (top);
\end{tikzpicture}
}\\
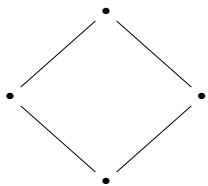
The Hasse diagram of the poset $\mathcal D_2$
\end{center}
Despite the simplicity of the diamond poset, the question of its induced saturation number is still open. The most recent bounds are $\sqrt n\leq\text{sat}^*(n, \mathcal D_2)\leq n+1$ proved by Martin, Smith and Walker \cite{martin2019improved}. Ferrara, Kay, Kramer, Martin, Reiniger, Smith and Sullivan \cite{Ferrara2017TheSN} conjectured that $\text{sat}^*(n, \mathcal D_2)=\Theta(n)$.\par In this paper we prove that $\text{sat}^*(n, \mathcal D_2)\geq (4-o(1))\sqrt n$, thus improving the constant factor. We remark that the bound of $\sqrt n$, proved in \cite{martin2019improved}, is the result of an argument about
the `local structure' of a diamond-saturated family, and in fact this type of argument cannot get beyond $\sqrt n$. To get beyond the $\sqrt n$ barrier and
achieve $4\sqrt{n}$, we develop a more `global' kind of argument which makes full use of the properties of minimal/maximal sets in a diamond-saturated family.\par Most importantly, our proof explores in depth what it means for a diamond-saturated family to be of size $c\sqrt n$ for a constant $c$. Surprisingly, such a structure is very rich in properties and yet, as far as we can see, there is no indication that such a family cannot exist. This suggests that perhaps the induced saturation number for the diamond in not of linear growth.
\section{The main result}
We first remark that, originally, the proof gave a $2\sqrt{2n}$ bound. To reach the current bound we needed to answer negatively Question 5. This has been added at the very end of the paper.

The following lemma, which is a special case of Lemma 9 in paper \cite{martin2019improved}, shows the importance of minimal elements in a $\mathcal D_2$-saturated family $\mathcal F$. Ferrara, Kay, Kramer, Martin, Reiniger, Smith and Sullivan \cite{Ferrara2017TheSN} proved that if $\mathcal F$ contains the empty set (or the full set $[n]$), then $|\mathcal F|> n$. This fact will be used repeatedly throughout the proof.
\begin{lemma}
Let $\mathcal F$ be a $\mathcal{D}_2$-saturated family. Let $S$ be a minimal element of $\mathcal F$. Then $|\mathcal F|\geq |S|$.
\end{lemma}
\begin{proof}
If $S$ is the empty set, then the statement is trivially true.\\Now we assume $S\neq\emptyset$, and for each element $i$ of $S$ we will find an element of $\mathcal F$ that contains all elements of $S$ except $i$. This will give us $|S|$ elements of $\mathcal F$, as desired.\par More precisely, by the minimality of $S$ we have that $S-\{i\}\notin\mathcal F$. Therefore, since $\mathcal F$ is diamond-saturated, $S-\{i\}$ will have to form a diamond when added to the family. We obtain three sets $A$, $B$ and $C$ of $\mathcal F$ such that they form a diamond together with $S-\{i\}$. By the minimality again we can only have $S-\{i\}$ the minimal element of the diamond. Let $A$ be the maximal element of the diamond.\par Suppose $B\neq S$ and $C\neq S$. If $i\in B$ and $i\in C$, then we observe that $A$, $B$, $C$ and $S$ form a diamond in $\mathcal F$, contradiction. Thus we can assume, without loss of generality, that $i\notin B$. So we have $S-\{i\}\subset B$ and $i\notin B$, as claimed.\par If on the other hand $C=S$, then $B$ and $S$ are incomparable and $S-\{i\}\subset B$. So we again obtain that $i\notin B$ and $S-\{i\}\subset B$, which finishes the proof. 
\end{proof}
\begin{theorem} For every $c<2\sqrt2$ there exists an $n_0$ such that $\text{sat}^*(n, \mathcal D_2)\geq c\sqrt n$ for any $n\geq n_0$.\end{theorem}\begin{proof}
\par Suppose for a contradiction that for some $c<2\sqrt2$ we have $\text{sat}^*(n, \mathcal D_2)\leq c\sqrt n$ for some arbitrarily large $n$.\par Fix $\mathcal F$ an arbitrary diamond-saturated family with cardinality at most $c\sqrt n$. This immediately implies that $\emptyset, [n] \notin\mathcal F$. Fix $S\in\mathcal F$ a minimal set with respect to inclusion. From Lemma 1 we know that $|S|\leq c\sqrt n$. Thus there exist $n-c\sqrt n$ singletons such that $i\notin S$. For those singletons, at least $n-2c\sqrt n$ of the sets $S\cup\{i\}$ are not in $\mathcal F$, by our initial assumption on the size of the family.\\A set $S\cup\{i\}$ that is not in our family must form a diamond with 3 elements of $\mathcal F$ by the saturation of $\mathcal F$. Assume $S\cup\{i\}, A, B, C$ form a diamond where $A, B, C\in \mathcal F$.\par We observe that $S\cup\{i\}$ cannot be the maximal element of such a diamond for more than $c\sqrt n$ singletons. Indeed, for each singleton $i$ for which $S\cup\{i\}$ is the maximal element of a diamond, let $V_i$ be minimal among the minimal elements of such diamonds. We observe that each $V_i$ is a minimal element of $\mathcal F$ and that $i\in V_i$ by the minimality of $S$ ($V_i\neq S$ as they have different sizes). Moreover, $V_i=S-K_i\cup\{i\}$ for some $K_i\subseteq S$. This implies that $S\cup\{i\}$ is the maximal element of the diamond for at most $c\sqrt n$ singletons.\par Also, $S\cup\{i\}$ cannot be the minimal element of the diamond because then $A, B, C, S$ would form a diamond in $\mathcal F$, contradicting the fact that $\mathcal F$ is diamond free. Thus, $S\cup\{i\}$ has to be one of the two incomparable elements of the induced diamond for at least $n-3c\sqrt n$ singletons. Therefore, for these singletons $i$ we have the structure below, where $A_i, B_i, S_i\in\mathcal F$.
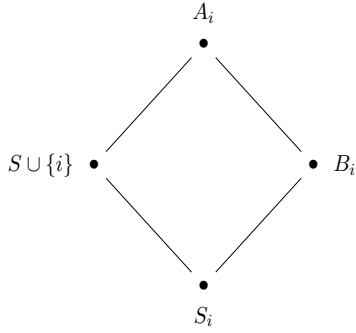
\begin{center}
\resizebox{5 cm}{4.5 cm}{
\begin{tikzpicture}
\node[label=above:$A_i$] (top) at (10,2) {$\bullet$}; 
\node[label=left:$S\cup\{i\}$] (left) at (8,0) {$\bullet$};
\node[label=right:$B_i$] (right) at (12,0) {$\bullet$};
\node[label=below:$S_i$] (bottom) at (10,-2) {$\bullet$};
\draw (top) -- (left) -- (bottom) -- (right) -- (top);
\end{tikzpicture}
}
\end{center} 
\par Moreover, we observe that by minimality either $S=S_i$ or $S_i-S=\{i\}$. If $S_i\neq S$ then there exists $K_i\subseteq S$ such that $S_i=S\cup\{i\}-K_i$, and all such $S_i$ are pairwise different because $S_i$ is the only one containing $i$. Therefore there are at least $n-4c\sqrt n$ singletons $i$ such that $S_i=S$. We will now focus on these singletons for which we have the following diamond, where $B_i$ is of maximal cardinality with respect to this construction, and $A_i$ is of minimal cardinality, after choosing $B_i$.
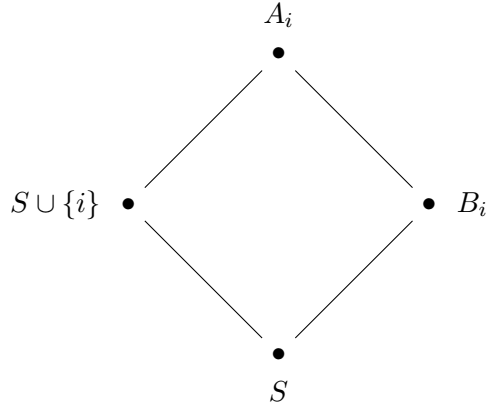
\begin{center}
\begin{tikzpicture}
\node[label=above:$A_i$] (top) at (10,2) {$\bullet$}; 
\node[label=left:$S\cup\{i\}$] (left) at (8,0) {$\bullet$};
\node[label=right:$B_i$] (right) at (12,0) {$\bullet$};
\node[label=below:$S$] (bottom) at (10,-2) {$\bullet$};
\draw (top) -- (left) -- (bottom) -- (right) -- (top);
\end{tikzpicture}
\end{center}
\par We observe that since $B_i$ and $S\cup\{i\}$ are incomparable, but $S\subset B_i$, then we must have $i\notin  B_i$.
\begin{claim}If $i\neq j$, then $B_i\cup\{i\}\neq B_j\cup\{j\}$.
\end{claim}
\begin{proof}
Suppose $B_i\cup\{i\}=B_j\cup\{j\}$. Since $i\notin B_i$ and $j\notin B_j$, we must have $i\in B_j$ and $j\in B_i$, which implies that $B_i$ and $B_j$ are incomparable. We now observe that $S, B_i, B_j, A_i$ form a diamond in $\mathcal F$, which is a contradiction. We can choose $A_i$ to be the maximal element because $B_j\subset B_j\cup\{j\}=B_i\cup\{i\}\subseteq A_i$.
\end{proof}
We deduce from Claim A and the assumption on the size of $\mathcal F$ that for at least $n-5c\sqrt n$ singletons $i$, $B_i\cup\{i\}$ is not in the family. Therefore, by saturation, each element $B_i\cup\{i\}$ has to form a diamond with 3 different elements of $\mathcal F$. Let $X_i, Y_i, N_i$ be three such elements. We notice that $B_i\cup\{i\}$ cannot be the minimal element because then $B_i, X_i, Y_i, N_i$ would form a diamond in $\mathcal F$. It also cannot be the maximal element of the diamond because $B_i\cup\{i\}\subset A_i$, thus $A_i, X_i, Y_i, N_i$ will again form a diamond in $\mathcal F$. We conclude that $B_i\cup\{i\}$ has to be one of the two incomparable elements as shown in the picture below. We choose $N_i$ of minimal cardinality with respect to this configuration.
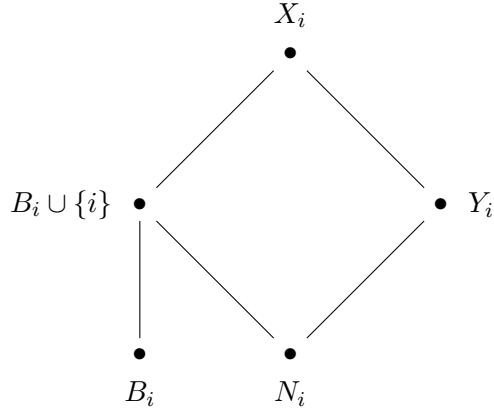
\begin{center}
\begin{tikzpicture}
\node[label=above:$X_i$] (top) at (10,2) {$\bullet$}; 
\node[label=left:$B_i\cup\{i\}$] (left) at (8,0) {$\bullet$};
\node[label=right:$Y_i$] (right) at (12,0) {$\bullet$};
\node[label=below:$N_i$] (bottom) at (10,-2) {$\bullet$};
\node[label=below:$B_i$] (extra) at (8,-2) {$\bullet$};
\draw (top) -- (left) -- (bottom) -- (right) -- (top);
\draw (extra)--(left);
\end{tikzpicture}
\end{center}
\begin{claim} $B_i$ and $N_i$ are incomparable.
\end{claim}
\begin{proof}
Suppose they are comparable. There are three cases:
\begin{enumerate}
\item $N_i\subset B_i$.\\ In this case, if $B_i$ and $Y_i$ were incomparable, then we would have the diamond $X_i, B_i, N_i, Y_i$ inside $\mathcal F$. Therefore $B_i$ and $Y_i$ are comparable and the only option is $B_i\subset Y_i$ as $B_i\cup\{i\}\parallel Y_i$. This also implies that $i\notin Y_i$. Finally we have $S\subset B_i\subset Y_i$, $S\cup\{i\}\subset B_i\cup\{i\}\subset X_i$, and $S\cup\{i\}\parallel Y_i$ since $i\notin Y_i$ and clearly $|S\cup\{i\}|\leq |B_i|<|Y_i|$. This means that we have the diamond below, which contradicts the maximality of $B_i$.
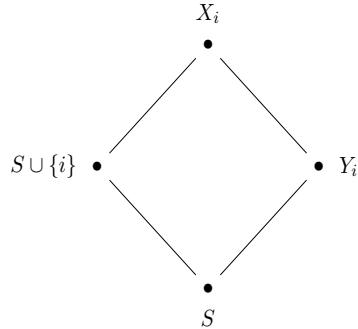
\begin{center}
\resizebox{5 cm}{4.5 cm}{
\begin{tikzpicture}
\node[label=above:$X_i$] (top) at (10,2) {$\bullet$}; 
\node[label=left:$S\cup\{i\}$] (left) at (8,0) {$\bullet$};
\node[label=right:$Y_i$] (right) at (12,0) {$\bullet$};
\node[label=below:$S$] (bottom) at (10,-2) {$\bullet$};
\draw (top) -- (left) -- (bottom) -- (right) -- (top);
\end{tikzpicture}
}
\end{center}
\item $N_i=B_i$.\\This case reduces to the previous case since now we already know $B_i\subset Y_i$.
\item $B_i\subset N_i$.\\This case is impossible by cardinality since $N_i\subset B_i\cup\{i\}$, thus $|B_i|<|N_i|<|B_i\cup\{i\}|=|B_i|+1$, a contradiction.
\end{enumerate}
\end{proof}
An immediate consequence of Claim B is that $S$ and $N_i$ are incomparable. If they were not, then $S\subset N_i$ would form the diamond $X_i, B_i, N_i, S$ in $\mathcal F$, contradiction. Lastly, $S$ cannot be equal to $N_i$ since $S\subset B_i$, but $B_i$ and $N_i$ are incomparable.\\We also remark that by the minimality of $N_i$ we have that any two distinct $N_i$ are incomparable, and that $i\in N_i$. The second remark follows from the fact that $N_i\parallel B_i$, but $N_i\subset B_i\cup\{i\}$.\par The following claim will be very useful for the construction and consequent modification of a certain bipartite graph at the end of the section.
\begin{claim} If $i\neq j$ then we cannot have both $N_i=N_j$ and $B_i=B_j$.
\end{claim}
\begin{proof} Suppose $N_i=N_j$. Then by previous remark we have that $i\in N_i$ and consequently $i\in N_j$. Also $N_j\subset B_j\cup\{j\}$, thus $i\in B_j$. On the other hand $i\notin B_i$, hence $B_i\neq B_j$ which finishes the claim.
\end{proof}
The next claim is not explicitly used in the proof, however it could be of potential interest towards proving more than the result in this paper, and it further illustrates how constraining and structurally rich is the property of being a minimal diamond-saturated family.
\begin{claim} If $B_i\neq B_j$, then $A_i\neq A_j$.
\end{claim}
\begin{proof} Assume for a contradiction that $B_i\neq B_j$ and $A_i=A_j$. We cannot have $B_i\parallel B_j$ because otherwise $A_i=A_j, B_i, B_j, S$ would form a diamond in $\mathcal F$. Therefore, without loss of generality, we can assume $B_i\subset B_j$ and we have the following diagram. 
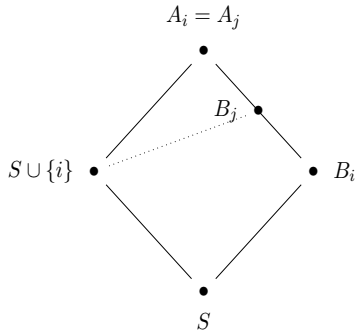
\begin{center}
\resizebox{5 cm}{4.5 cm}{
\begin{tikzpicture}
\node[label=above:${A_i=A_j}$] (top) at (10,2) {$\bullet$}; 
\node[label=left:$S\cup\{i\}$] (left) at (8,0) {$\bullet$};
\node[label=right:$B_i$] (right) at (12,0) {$\bullet$};
\node[label=below:$S$] (bottom) at (10,-2) {$\bullet$};
\node[label=left:$B_j$] (extra) at (11,1) {$\bullet$};
\draw (top) -- (left) -- (bottom) -- (right) -- (top);
\draw[dotted] (extra)--(left);
\end{tikzpicture}
}
\end{center}
If $S\cup\{i\}$ is not comparable to $B_j$, then the diamond formed by these two, $A_i=A_j$ and $S$ would contradict the maximality of $B_i$. Hence we need to have $S\cup\{i\}$ and $B_j$ comparable, and by cardinality (there is no set strictly between $S$ and $S\cup\{i\}$), the only possibility is $S\cup\{i\}\subset B_j$.\\But now we have the diamond formed by $S, B_i, S\cup\{i\}, B_j$ which contradicts the minimality of $A_i$ with respect to $B_i$.
\end{proof}
\par We have already seen previously that $B_i\cup\{i\}\neq B_j\cup\{j\}$ for $i\neq j$. Thus we have at least $n-5c\sqrt n$ sets, $B_i\cup\{i\}\notin\mathcal F$. Each of them has a corresponding set $N_i$, although the $N_i$'s can sometimes coincide for different $i$'s. We build the following bipartite graph: the vertex set is $\mathcal B\sqcup\mathcal N$, where $\mathcal B$ consists of the sets $B_i\cup\{i\}$ and $\mathcal N$ consists of the corresponding sets $N_i$. The only edges are the ones joining $B_i\cup\{i\}$ to the corresponding $N_i$ for each $i$. We observe that each vertex in $\mathcal B$ has degree 1, thus we have at least $n-5c\sqrt n$ edges.\par We now modify the graph by replacing $B_i\cup\{i\}$ with $B_i$ for all $i$, and identifying the same repeating set with a single vertex -- in other words, if $B_i=B_j$ for two $i\neq j$, the vertex $B_i\cup\{i\}$ and the vertex $B_j\cup\{j\}$ will both be identified with the vertex $B_i=B_j$. This new graph, which we call $G$, is bipartite and has vertex set $\mathcal B'\sqcup\mathcal N$, where $\mathcal B'$ consists of the sets $B_i$. Notice that no vertex in $\mathcal B'$ appears in $\mathcal N$ since all the $B_i$ contain $S$, while all the $N_i$ are incomparable to $S$.\par If an edge were to contract, that would mean that for two different $i$ and $j$ we have $N_i=N_j$ and $B_i=B_j$, which contradicts Claim C. Hence, the modified graph still has at least $n-5c\sqrt n$ edges.\par Assume that $|\mathcal N|=k$ and that $d$ is the biggest degree in $\mathcal N$. Since $G$ is bipartite we have that the number of edges is the sum of degrees in $\mathcal N$ which is less or equal to $kd$. Thus we have that $n-5c\sqrt n\leq kd\Rightarrow d\geq\dfrac{n-5c\sqrt n}{k}$. This also tells us that the size of $\mathcal B'$ is at least $d\geq\dfrac{n-5c\sqrt n}{k}$. Moreover, we already have that $|\mathcal F|\geq|\mathcal B'|+|\mathcal N|\geq k+\dfrac{n-5c\sqrt n}{k}\geq2\sqrt{n-5c\sqrt n}$.
\par Similarly, by looking at $\mathcal G=\{A: \bar A \in\mathcal F\}$, where $\bar A$ is the complement of $A$ in $[n]$, we observe that this is also a diamond-saturated family of the same size as $\mathcal F$, where the minimal elements are the complements of the the maximal elements of $\mathcal F$. We can do the same analysis as above by fixing $T$ a minimal element of $\mathcal G$, and construct a bipartite graph $H$ analogously to the above graph $G$.\par The bipartite graph $H$ has vertex set $\mathcal C'\sqcup\mathcal M$, where $\mathcal M$ consists of the minimal elements, denoted by $M_i$ (equivalent to the $N_j$), and $\mathcal C'$ consists of the elements that contain $T$, denoted by $C_k$ (equivalent to the $B_l$). Therefore $\bar{M_i}\in \mathcal F$ are maximal elements in $\mathcal F$. On the other hand, any $N_j$ is not a maximal element as it is contained in $Y_j$ by construction. Similarly, any $B_l$ is not a maximal element as it is contained in $A_l$. We conclude that no $\bar{M_i}$ can be equal to any $N_j$ or to any $B_l$. Moreover, $B_l$ is neither a maximal nor a minimal element in $\mathcal F$ since it is between $S$ and $A_i$, thus no $\bar{C_k}$ is a minimal or a maximal element either, which implies that no $\bar{C_k}$ is equal to any $N_j$.\par Let $|\mathcal M|=t$. By the same argument as above, now applied to the graph $H$, we have that $|\mathcal C'|\geq\dfrac{n-5c\sqrt n}{t}$. Let $\mathcal M^c=\{\bar{A}:A\in\mathcal M$\} and $\mathcal C'^c=\{\bar{A}:A\in\mathcal C'\}$, thus $\mathcal M^c$ and $\mathcal C'^c$ are subsets of $\mathcal F$. As observed above $\mathcal N\cap\mathcal M^c=\emptyset$, $\mathcal N\cap\mathcal C'^c=\emptyset$ and $\mathcal M^c\cap\mathcal B'=\emptyset$. Therefore we have that $|\mathcal F|\geq |\mathcal N|+|\mathcal M|+|\mathcal B'\cup\mathcal C'^c|$. Assume without loss of generality that $t\geq k$. We then have that $|\mathcal F|\geq k+t+|\mathcal B'|\geq 2k+\dfrac{n-5c\sqrt n}{k}\geq2\sqrt{2(n-5c\sqrt n)}$, which is greater than $c\sqrt n$ for $n$ large enough, a contradiction.
\end{proof}
The above proof leads to a natural question, namely how disjoint can $\mathcal B'$ and $\mathcal C'^c$ be? Suppose the family $\mathcal F$ contains a minimal element $P$ and a maximal element $R$  such that $P$ and $R$ are not comparable. It turns out that in that case $\mathcal B'$ and $\mathcal C'^c$ can be disjoint, thus giving $|\mathcal F|\geq|\mathcal N|+|\mathcal B'|+|\mathcal M^c|+|\mathcal C'^c|\geq k+\dfrac{n-5c\sqrt n}{k}+t+\dfrac{n-5c\sqrt n}{t}\geq4\sqrt{n-5c\sqrt n}$. Indeed, we run the above argument once for $\mathcal F$ with fixed minimal element $P$ and once for $\mathcal G$ with fixed minimal element $\bar{R}$. Now we notice that if $B_l=\bar{C_k}$ for some $l$ and $k$, then $S\subset B_l=\bar{C_k}\subset \bar{\bar R}=R$, a contradiction.\par However, whether such minimal and maximal sets exist in any diamond-saturated family (without $\emptyset$ and $[n]$) seems to be a non-trivial question. Using the above notation, the following proposition guarantees the existence of $P$ and $R$ under some mild assumptions.
\begin{proposition}
Suppose there is no $i\notin S$ such that $S\cup\{i\}\notin\mathcal F$, $S=S_i$ and $B_i\cup\{i\}\in\mathcal F$. Then there exists a minimal element $P$ and a maximal element $R$ in $\mathcal F$ such that $P$ and $R$ are not comparable.
\end{proposition}
\begin{proof}
We begin by noticing that all the $N_i$ and $S$ are minimal elements in $\mathcal F$ and $i\in N_i$ for every $i$ such that $i\notin S$ and $S-K_i\cup\{i\}\notin \mathcal F$ for any $K_i\subseteq S$. For the singletons $i\notin S$ such that $S-K_i\cup\{i\}\in\mathcal F$ for some $K_i\subseteq S$, we consider $L_i\in \mathcal F$ to be the minimal element such that $L_i\subseteq S-K_i\cup\{i\}$. Because $S$ is itself a minimal element, we need to have $i\in L_i$.\\If every maximal element is comparable to every minimal element, then any maximal element must contain the union of all $N_i$, $L_i$ and $S$ which is $[n]-W$, where $W=\{i:i\notin S, S\cup\{i\}\in \mathcal F\}$.\\Suppose that $|W|\geq 2$. If $S\cup\{i\}$ and $S\cup\{j\}$ are in $\mathcal F$ for $i\neq j$ and $i, j\notin S$, then we cannot have a maximal element $T\in \mathcal F$ that contains the pair $\{i,j\}$ since by assumption $S\subset T$ and $S, S\cup\{i\}, S\cup\{j\}$ and $T$ will form a diamond in $\mathcal F$. This implies that no element of $\mathcal F$ contains both $i$ and $j$ (as every element is included in a maximal element), in particular $\{i, j\}\notin\mathcal F$. Thus $\{i, j\}$ will have to form a diamond with 3 elements of $\mathcal F$ by saturation. However, since the empty set is not in the family, $\{i,j\}$ cannot be the maximal element of the diamond, therefore it will be a subset of one of the three sets it form a diamond with, contradiction.\\We conclude that $|W|\leq 1$, thus the union of all the minimal sets is either $[n]$ or $[n]-\{i\}$ for some $i\notin S$ and $S\cup\{i\}\in \mathcal F$. If the latter is true, then take $D$ to be the maximal element of $\mathcal F$ that contains $S\cup\{i\}$. $D$ will have to contain $[n]-\{i\}$ too by our assumption, thus $D=[n]$. We see that in both cases we have to have $[n]\in\mathcal F$, which is a contradiction. 
\end{proof}

\section{Could $\text{sat}^*(n,\mathcal D_2)$ be $O(\sqrt n)$?}
\par The above proof explores the extremal behaviour of a diamond-saturated family of size $O(\sqrt n)$. The square root bound appears to push the minimal and maximal elements closer together, yet spread through most of the layers of the hypercube -- note that this is quite unlike the two canonical examples of diamond-saturated families, namely a chain of size $n+1$, and the family of all singletons and the empty set. Indeed, consider the graphs constructed towards the end of the proof. They show that, under the condition that the diamond-saturated family $\mathcal F$ is of square root order, $\mathcal F$ must be in a way invariant under taking complements -- for example, the antichains formed by the minimal and maximal elements have to roughly look the same and be of $\sqrt n$ order.\par It is clear from the proof that if the induced saturation number for the diamond is $\Theta (\sqrt n)$, then the size of the biggest antichain of a family of this size has to be of $\sqrt n$ order. This can be seen by looking at the bipartite graph considered in the proof and the fact that one side, namely the $N_i$'s, is an antichain of size $k$. Indeed, we have that the number of edges is equal to the sum of the degrees of the $N_i$, thus $n-5c\sqrt n\leq k\times\text{max degree}$. Because for each $i$, $i\in N_i$, we have $\text{deg}(N_i)\leq|N_i|\leq|\mathcal F|\leq c\sqrt n$, where the middle inequality comes from the fact that the $N_i$ are minimal elements. Therefore the maximum degree is at most $c\sqrt n$, hence $k\geq\dfrac{\sqrt n}{c}-5$.\par By applying Dilworth's theorem, we get that $\mathcal F$ can be decomposed into roughly $\sqrt n$ chains. Since $|\mathcal F|=O(\sqrt n)$, this suggests a family of $c'\sqrt n$ disjoint chains, each of constant size and, more importantly, positioned in such a way that there are no common interior gaps for all of them. We believe that this is possible and that the above proof has some of the key clues to construct such a diamond-saturated family. More precisely, we conjecture the following.
\begin{conjecture} $\text{sat}^*(n, \mathcal D_2)=\Theta(\sqrt n)$. Moreover, there exists a constant $c$ such that for $n$ large enough there exists a diamond-saturated family $\mathcal F$ consisting of $\sqrt n$ chains each of size $c$, with the property that if $C_i\subseteq A\subseteq B_i$, where $C_i$ and $B_i$ are elements of the $i^\text{th}$ chain for every $i$, then $A\in\mathcal F$.
\end{conjecture}
\par Led on by the above analysis, we can also ask the following question.
\begin{question}
Let $\mathcal F$ be a diamond-saturated family that does not contain $\emptyset$ or $[n]$. Can all the minimal elements of $\mathcal F$ be subsets of all the maximal elements of $\mathcal F$?
\end{question}
\textbf{Acknowledgement.} I would like to thank my supervisor Professor Imre Leader and the anonymous referees for guidance and helpful comments.\\

\vspace{1em}
\textbf{Update. }It turns out that the answer to Question 5 is `no', and the argument  is rather simple. Indeed, suppose that there exists a diamond-saturated family $\mathcal F$ with ground set $[n]$ such that $\emptyset, [n]\notin\mathcal F$, and every minimal element is comparable to every maximal element. Let $A_1, A_2,\dots,A_t$ be the minimal elements, and $B_1, B_2,\dots,B_l$ the maximal elements. Let also $X=\cup_{i=1}^t A_t$. By our assumption we have that $X\subseteq B_i$ for all $i\in[l]$. We now note that $[n]\setminus X$ is incomparable to all $A_i$ and $B_j$, for all $i\in [t]$ and all $j\in[l]$. This means that, not only it is not a member of $\mathcal F$, but it is incomparable to every set in $\mathcal F$. Consequently, this implies that it does not form a diamond added when added to $\mathcal F$, a contradiction. As explained above, this implies that $\text{sat}^*(n,\mathcal D_2)\geq(4-o(1))\sqrt n$.
\bibliographystyle{amsplain}
\bibliography{document}
\Addresses
\end{document}